\documentclass[12pt,reqno]{amsart}
\usepackage{mathrsfs}
\usepackage{amsgen}
\usepackage{amscd}
\usepackage{latexsym}
\usepackage{amsthm}
\usepackage{graphicx}
\usepackage{color}
\usepackage{enumerate}
\usepackage{enumitem}
\usepackage{amssymb,amsmath,graphicx,amsfonts,euscript}
\usepackage{color}
\usepackage{empheq}
\usepackage{cases}
\usepackage{comment}
\usepackage{hyperref}
\usepackage{todonotes}

\usepackage{setspace} 
\usepackage[margin=3cm, a4paper]{geometry}

\def\H{{\cal H}}

\def\N{\mathbb{N}}
\def\R{\mathbb{R}}
\def\Z{\mathbb{Z}}

\def\C{\mathbb{C}}

\def\H2{H^2(\R^N)}
\def\L2{L^2(\R^N)}
\def\to{\rightarrow}

\def\H{{\cal H}}

\def\H1{H^1(\R)}

\newcommand{\om}{\omega}

 \newcommand{\ft}{{\mathcal{F}}}
 \newcommand{\Ga}{\Gamma}

\newcommand{\ls}{\lesssim}

 \newcommand{\Del}[1]{}

\numberwithin{equation}{section}
\newtheorem{thm}{Theorem}[section]
\newtheorem{cor}[thm]{Corollary}
\newtheorem{lem}[thm]{Lemma}
\newtheorem{prop}[thm]{Proposition}

\theoremstyle{definition}
\newtheorem{remark}[thm]{\textbf{Remark}}
\newtheorem*{exam*}{Examples}

\makeatletter
\@namedef{subjclassname@2020}{\textup{2020} Mathematics Subject Classification}
\makeatother

\begin{document}

\setcounter{page}{1}

\title[Solitary waves to fKdV]{A direct construction of solitary waves for a fractional Korteweg-De Vries equation with an inhomogeneous symbol}

\author{Swati Yadav}
\address{Department of Mathematical Sciences, Norwegian University of Science and
Technology, 7491 Trondheim, Norway}
\email{swati.yadav@ntnu.no, yswati.iitbhu@gmail.com}
\thanks{}

\author{Jun Xue}
\address{Department of Mathematical Sciences, Norwegian University of Science and
Technology, 7491 Trondheim, Norway}
\email{jxuemath@hotmail.com}
\thanks{}


\begin{abstract}\noindent
We construct solitary waves for the fractional Korteweg--De Vries type equation
\begin{align*}
u_t + (\Lambda^{-s}u + u^2)_x = 0,
\end{align*}
where  $\Lambda^{-s}$ denotes the Bessel potential operator $(1 + |D|^2)^{-\frac{s}{2}}$ for  $s \in (0,\infty)$. The approach is to parameterise the known periodic solution curves through the relative wave height. Using a priori estimates, we show that the periodic waves locally uniformly converge to waves with negative tails, which are transformed to the
desired branch of solutions. The obtained branch reaches a highest wave, the behavior of which varies with  $s$. The work is a generalisation of recent work by Ehrnström--Nik--Walker, and is as far as we know the first simultaneous construction of small, intermediate and highest solitary waves for the complete family of (inhomogeneous) fractional KdV equations with negative-order dispersive operators. The obtained waves display exponential decay rate as \(|x| \to \infty\). 
\end{abstract}

\maketitle

\section{Introduction}
This work focuses on the direct construction of a full family of solitary waves for a fractional Korteweg--de Vries (fKdV) equation of the form
\begin{align}
u_t + (\Lambda^{-s}u + u^2)_x = 0. 	\label{eq-fKdV}
\end{align}
Here, $u(t, x)$ is a real-valued function representing the deflection of a water fluid surface from its rest position 
at time $t \ge 0$ and position $x \in \R$. The symbol $\Lambda^{-s}$  is a Fourier multiplier operator of order $-s$ defined through
\begin{align*}
	\Lambda^{-s}\colon \quad f \mapsto  \mathcal F^{-1} 
	\big[    (1 + \xi^2)^{-\frac{s}{2}} \hat f(\xi)       \big],
\end{align*}
for all positive values of $s$. Sometimes referred to as a Bessel potential operator, this operator can also be expressed via convolution as
$\Lambda^{-s} f = K_s * f$, where the convolution kernel $K_s$ is given by
\begin{align}
	K_s(x) = \ft^{-1} \big(  (1 + \xi^2)^{-\frac{s}{2}}  \big)(x)
=\frac{1}{2\pi} \int_{\mathbb R} (1 + \xi^2)^{-\frac{s}{2}}  e^{ix \xi} \, d\xi.
\label{kernel-Ks}
\end{align}
The last equality defines the Fourier transform. We shall consider travelling-wave solutions $u(x, t) = \varphi(x - \mu t)$, where $\mu$ is the (rightward) wave-speed. In this setting, equation  \eqref{eq-fKdV} reads
\begin{align*}
	-\mu \varphi' + (\varphi^2)' + (\Lambda^{-s}\varphi)' = 0.
\end{align*}
By integrating the above equation, we obtain the steady equation
\begin{align}
	-\mu \varphi + \varphi^2 + \Lambda^{-s}\varphi = 0, \label{eq:steady}
\end{align}
where the right-hand side can be set to zero by the Galilean transformation
\begin{align*}
\varphi \mapsto \varphi + \rho,\quad \mu \mapsto \mu + \rho,   
\end{align*}
with $\rho$ chosen such that $\rho(1 - \mu - \frac{\rho}{2})$ cancels the constant of integration.

The equation \eqref{eq-fKdV} is an \emph{inhomogeneous} fractional KdV equation, and therefore a Whitham type equation \cite{WhithamWaterWaves}. It originates from the KdV equation by replacing the polynomial dispersion relation \(m_\text{KdV}(\xi) = 1 - \frac{1}{6}\xi^2\) expressed by the Fourier multiplier \(1 + \frac{1}{6}\partial_x^2\) with more exact approximations of the Euler dispersion relation $m_{\text{Euler}}(\xi)= \sqrt{\frac{\tanh{\xi}}{\xi}}$, or, as in the case of \eqref{eq-fKdV}, with full scales of equations describing varying strength and behaviour of the dispersion with respect to nonlinear effects. In the case of the Whitham equation, the symbol is exactly \(m_{\text{Euler}}\), describing the right-going branch of the dispersion relation for gravity water waves on finite depth. The modelling of bi-directional motion can be captured in Boussinesq equations \cite{Broer-Boussinesq-64}, including full-dispersion Whitham--Boussinesq systems with symbol $m_{\text{Euler}}^2(\xi)= \frac{\tanh{\xi}}{\xi}$, where the linear wave speed is the positive or negative root of this expression \cite{Lannes-Book-2013}. Many such models quantitatively well describe the propagation of small-amplitude waves in shallow water over moderate time-spans \cite{MR4321411}.

A direct construction of a full set of solitary waves for the Whitham equation with symbol $m_\text{Euler}$ was recently given in \cite{MR4531652}. In difference to the KdV equation, the Whitham equation features dispersion that allows both for solitary, breaking and highest waves. Ehrnstr\"{o}m and Wahl\'{e}n \cite{EhrnstromWahlen-AIHAN-19} were the first to prove that there exist periodic travelling wave solutions to the Whitham equation having exactly $1/2-$H\"{o}lder regularity. In the more recent paper \cite{TruongWahlen-arxiv-20}, the authors constructs global curves of solitary waves up to the highest wave for the same problem using a new nonlocal variant of the center manifold theorem, and global bifurcation. Whereas small solitary solutions for \eqref{eq-fKdV} was known from \cite{ehrnstrom2012existence,Hildrum2020small,Atanas-Wright}, global curves of solitary solutions up the highest wave have not been proved, and we construct them using the direct method of \cite{MR4531652}.

In the periodic setting of \eqref{eq-fKdV} more is known. Ørke in \cite{Orke-arxiv-22} established existence and regularity of the highest periodic waves for fKdV equation \eqref{eq-fKdV} when $s \in (0,1)$, and proved that the non-constant solutions (say $\varphi$) to the steady equation are everywhere smooth except where the wave-height equals half the wave-speed, $\varphi(0)=\mu/2$. As in the Whitham case, the highest periodic waves show $s-$H\"{o}lder regularity at the crest attained in the form of cusp. Le \cite{Le-AA-22} studied a similar model for the even weaker dispersion $s>1$, and proved that there the highest waves are Lipschitz continuous. Arnesen \cite{MR3987024} considered  the Degasperis--Procesi equation with symbol $(1 + \xi^2)^{-1}$ and used the nonlocal method to prove the existence of highest travelling-wave solutions. As we shall prove, when $s = 1$, highest waves are not Lipschitz, but display of cusp of logarithmic Lipschitz regularity near the origin. Similar behavior was observed by Ehrnstr\"{o}m et al. \cite{ehrnstrom2019existence} in the bidirectional Whitham equation, which has symbol \(m_\text{Euler}^2\), which corresponds to \(s=1\). Very recently, Ehrnstr\"{o}m, M\ae hlen, and Varholm \cite{Ehrnstrom_Ola_Kristo} completed this investigation by providing exact leading-order asymptotics for the highest waves to this and more general non-local equations.  In summary, the regularity of the highest periodic wave solutions varies on $s$ for values $s<1$;  for $s = 1$ it is log-Lipschitz; while for $s>1$ it is always exactly Lipschitz. The value of $s$ however does affect the angle, or `cusple', at the wave crest.

Other authors have  investigated various generalizations of KdV equation with \emph{homogeneous} symbols. The homogeneous equation corresponding to \eqref{eq-fKdV} is
\begin{align}\label{homogeneous}
u_t + |D|^{-s}u_x + uu_x = 0,
\end{align}
with $|D|^{s}$ given by \(\mathcal F (|D|^{s}f)(\xi) = |\xi|^s \hat f(\xi)\). Bruell and Dhara \cite{Bruell-Dhara-2018} showed the existence and regularity of periodic travelling-wave solutions of \eqref{homogeneous} when $s > 1.$ 
{Dahne and Gómez-Serrano \cite{MR4626005} studied the same for $s=1$  with an impressive combination of analysis and interval arithmetics.
Later, Dahne \cite{dahne2024highest} was able to show the regularity and existence of the highest travelling wave solutions for the whole family $s \in (0,1)$, with exact leading asymptotics at the origin (matching that of \cite{Ehrnstrom_Ola_Kristo} in comparable cases)}. Hildrum and Xue \cite{Hildrim-Xue-2023} earlier studied  \eqref{homogeneous} for the same interval of $s$, but for more general nonlinearities, in some cases finding waves with simultaneous singularities at both the crests and troughs.

The purpose of this study is to establish existence of a full family of solitary waves for the fKdV equation \eqref{eq-fKdV}  with an inhomogeneous symbol $\Lambda^{-s}$, for the full range of values $s \in \R^{+}$. This contrasts with the homogeneous case \eqref{homogeneous}, which does not allow for solitary waves in $H^1$ in the regime we are considering here \cite{MR3188389}. {For $\Lambda^{s}$ with positive $s$, solitary waves, when they exist, can be constructed variationally with arbitrary height (see \cite{MR3763731,MR4097537,marstrander2023solitary}).
In this paper, we follow the same method as in \cite{MR4531652}, but have to tale care of the separate cases  \(s \in (0,1)\), \(s = 1\) and \(s > 1\), where the latter case is entirely new. The approach is built upon using a priori bounds obtained from the periodic theory, and in some cases we perform this theory for the first time. }

The organisation of the rest of the paper is as follows: In Section \ref{sec-kernel} we study the convolution kernel $K_s$, and introduce periodised kernels $K_{s,P}$. 
In Section \ref{Solution_properties}, we discuss the properties of the solution of \eqref{eq:steady} including its regularity at the highest point. We construct a sequence of periodic waves in Section \ref{sec-construction} showing the convergence to the travelling waves which will eventually be transformed into solitary waves. 
{In the last subsection, we discuss the decay rate of the obtained solitary waves. Here, we mention the work by Arnesen \cite{MR4324298}, which more generally considers the decay and symmetry of nonlocal dispersive wave equations.}

\section{Bessel Potential and Convolution Kernel}\label{sec-kernel}

In this section, we discuss some properties, estimates, and findings related to the Bessel potential operator $\Lambda^{-s}$ and associated kernel $K_s$. The kernel $K_s$, defined by \eqref{kernel-Ks}, is a widely recognized inverse Fourier transform of $(1+\xi^2)^{-\frac{s}{2}}$ and is given by 
\begin{align*}
	K_s(x) = \frac{1}{2^{\frac{s-1}{2}} \sqrt{\pi} \Ga(\frac{s}{2}) } {G}_{\frac{1-s}{2}} \big( |x| \big) |x|^{\frac{s-1}{2}}, \quad \text{for} \ s>0,
\end{align*} 
where ${G}_{\frac{1-s}{2}}$ denotes the modified Bessel function of second kind \cite[Chap. 2, Sec. 4]{aronszajn1961theory}, which provides the rationale for referring to $\Lambda^{-s}$ as a Bessel potential operator. {This operator serves as a generalization of the Riesz potential function, which
has a limitation that the `$s$' must be less than the dimension $n$ of the Euclidean space. In contrast, $\Lambda^{-s}$ operates effectively for all positive values of $s$.} 
The majority of the properties concerning the kernel $K_s$ and the operator $\Lambda^{-s}$ discussed here can be found in the works of Aronszajn and Smith \cite[Chap. 2]{aronszajn1961theory} and Grafakos \cite[Sec. 1.2.2]{Grafakos-Book-09}.
The first part of this section focuses on the convolution kernel $K_s$ which will be useful in proving the results for the operator $\Lambda^{-s}$. However, we shall refrain from presenting most of the proofs as similar results have been discussed in \cite{Le-AA-22,Orke-arxiv-22}. 

We say $X \lesssim Y$ whenever $X \leq c Y$ for some positive constant $c$ and $X \simeq Y$ when we have $ Y \lesssim X \lesssim  Y$. We use symbols $\lesssim_s$ and $\simeq_s$ when the constants depend on a variable $s$.
$O(Y)$ represents any quantity $X$ for which $|X| \lesssim Y$, while $o(Y)$ indicates $X / Y \to 0$ as $Y \to 0$. 
The general definitions of the sets and spaces mentioned in this paper are taken from the book by Triebel \cite[chap. 1]{Triebel-Book-92}.

\begin{lem}\label{lem-K-1} \cite{Grafakos-Book-09}
The kernel $K_s$ can be written in the following integral expression
\begin{align*}
	K_s(x) = \frac{1}{\sqrt{4\pi}\Gamma(\frac{s}{2})} \int_{0}^{\infty}
 e^{-t- \frac{x^2}{4t}} t^{\frac{s-3}{2}}
\,\mathrm dt,
\end{align*}
provided $s>0$. $K_s$ is even, positive, continuous, and smooth on $\mathbb R \backslash \{0\}$. $K_s$ is decreasing for $x > 0$ and integrable such that $\|K_s\|_{L^1(\mathbb R)} = 1$.
Moreover, it satisfies the following asymptotic estimates:

\begin{enumerate}[label=\roman*.]
    \item   for $|x| \ge 1$,
\begin{align*}
	K_s(x) \lesssim_s e^{-|x|},
\end{align*}

\item for $|x| \leq 1$,
\begin{eqnarray}\label{grafokos_eq}
    K_s(x) \simeq_s \left\{\begin{array}{lr} |x|^{s-1} - C_s + O(|x|^{s+1}), & 0<s<1\\
    \log \frac{1 }{|x|} + c_s - O(|x|^{2-\epsilon}), & s = 1 \\
     1 - o(|x|), & s>1
    \end{array}\right.,
\end{eqnarray}
where $\epsilon$ is a small positive real number. $C_s$ and $c_s$ are some positive constants depending on $s$.

\item when $|x| \ge 1$, we can define $K_s'$ as the first derivative of $K_s$, and 
\begin{align*}
	K'_s(x) \lesssim_s |x|e^{-|x|}.
\end{align*}
\end{enumerate}

\end{lem}

\noindent The function $K_s(x)$ is completely monotone on $(0, \infty)$ for $0 < s \leq 2$, that is
\begin{align*}
	\left(-\frac{d}{dx}\right)^j K_s(x) > 0, \quad  \text{for} \quad j = 0, 1, 2, \dots .
\end{align*}
The above statement is proved for $s\in (0,1)$ in \cite{Orke-arxiv-22}, but following the Lemma 2.12 in \cite{ehrnstrom2019existence}, the result holds for all $s$ up to 2.
$K_s$ belongs to the Sobolev space $W^{\alpha,1}$ of absolute integrable functions $h(x)$ such that $\ft^{-1} \big[  (1 + \xi^2)^{\frac{\alpha}{2}} \hat h(\xi)  \big]$ are also absolutely integrable for all $\alpha < s$ .
{ More specifically, if $s > 1$, $K_s$ is in the space $W^{1,1}$, i.e., the first derivative (in the distributional sense) of absolute integrable functions is also absolutely integrable}.  If
 $1 < s < 2 $, $K_s$ is $\beta-$H\"older continuous for $\beta \in (0, s - 1)$, Lipschitz continuous when $s = 2$, and continuously differentiable if $s > 2$. 


We define periodised kernel $K_{s,P}$ as 
\begin{align}
	K_{s,P}(x) = \sum_{n\in \mathbb Z}K_{s}(x+nP), \quad P \in (0,\infty), \label{eq:periodisedKernel}
\end{align}
assuming that $K_{s,P} $ approaches $ K_s$ as $P \to \infty$. Since, $K_s$ is integrable with $\|K_s\|_{L^1(\mathbb R)} = 1$ and is monotonically decreasing in $(0,\infty)$, (\ref{eq:periodisedKernel}) is well defined and absolutely convergent. We can also define  $K_{s,P}$ in terms of Fourier series 
\begin{align*}
	K_{s,P}(x) = \sum_{k \in \mathbb Z}C_k \exp\left(\frac{2\pi i k x}{P} \right),
\end{align*}
where
\begin{align*}
 C_k = \frac{1}{P} \int_{0}^{P} K_{s,P}(x) \exp\left(\frac{- 2 \pi i k x}{P} \right) \,\mathrm dx.
\end{align*}
$C_k$ can be calculated by substituting $K_{s,P}$ from (\ref{eq:periodisedKernel}), and computing integral using change of variable. Then we can rewrite $K_{s,P}$ explicitly as
\begin{align}\label{Fourier-series}
	K_{s,P}(x) = \frac{1}{P} \sum_{k \in \mathbb Z} \left( 1 + \big(  \frac{2\pi k}{P}  \big)^2  \right)^{- \frac s2} \exp \left( \frac{2\pi i k x}{P} \right),
\end{align}
which converges absolutely for $s>1$.

The definition of $K_{s,P}$ allows it to inherit many properties from the convolution kernel $K_{s}$ , but restricted to a specific domain. The following lemma lists these properties:

\begin{lem}\label{lem-K-2}
$K_{s,P}$ is even, positive, periodic with period $P$, and strictly increasing in $(-P/2, 0)$. For $0 < s \leq 2$, $K_{s,P}$ is completely monotone on $(0, P/2)$ and {convex on $(0,P)$}. {For $ s > 1$, $K_{s,P}$ is in $W^{1,1}$, in particular, $K'_{s,P}$ is integrable. }
\end{lem}
{\theoremstyle{remark}
\begin{remark}\label{Remark1}
\noindent The behavior of $K_{s,P}$ near the origin is similar to that of $K_{s}$  (Lemma \ref{lem-K-1}(\textit{ii.})). If we split $K_{s,P}$ in its regular and singular parts as $K_{s,P} = R_{s,P} + S_{s,P}$, where $R_{s,P}$ is a real analytic function in $(-P/2,P/2)$ and $S_{s,P}$ has the singularity at the origin, then for $s > 0$ and $|x| \leq 1$, we have
\begin{eqnarray*}
    S_{s,P}(x) \simeq_s \left\{\begin{array}{lr} |x|^{s-1}, & s \neq 2k-1\\
    |x|^{s-1}\log \frac{1}{|x|}, & s = 2k-1 
    \end{array}\right. , \quad k \in \Z^+ .
\end{eqnarray*}
\end{remark}
}
The following results on $\Lambda^{-s}$ are direct consequences of the fact that the kernels $K_{s}$ and $K_{s,P}$ are even and positive on the real line, and therefore on the interval $(-P/2,P/2)$ for $P>0$. 

\begin{prop} \label{Prop_bessel_2}
   (i) The operator $\Lambda^{-s}$ is parity preserving and strictly monotone on $C(\R)$.

 \noindent (ii) Let a function $f$ is continuous, odd, $P$-periodic, and non-negative on $(-P/2,0)$ such that for some $\bar x \in (-P/2,0)$, $f(\bar x)$ is positive, then $\Lambda^{-s}f>0$ on $(-P/2,0)$.
\end{prop}
The operator $\Lambda^{-s}$ is a classical symbol of order $-s$, that is, it satisfies the following inequality
    \begin{align*}
	  | D^k_\xi(1 + \xi^2)^{-\frac{s}{2}} | \lesssim_{s,k} (1 + |\xi|)^{-s-k}. 
\end{align*} 

\noindent Let $B^{m}_{p,r}(\mathbb R)$ represents a class of Besov spaces with $m \in \mathbb R $, $1 \le p,r \le \infty$. 
{According to \cite[Sec. 2.5]{MR2768550}, $L^\infty(\mathbb R)$ is embedded in $B^{0}_{\infty,\infty}(\mathbb R)$.} 
Moreover, 
$B^{s}_{\infty,\infty}(\mathbb R)$ are equivalent to the Zygmund spaces $\mathcal{C}^{s}(\mathbb R)$ 
of bounded functions {for any positive real number $s$}. 
$\mathcal{C}^{s}(\mathbb R)$ are equivalent to the H\"{o}lder spaces $C^{[s],s-[s]}$ for $s \in \R^+\backslash \N$ and are thus often referred to as the H\"{o}lder-Zygmund spaces. Additionally,  $\mathcal{C}^{s}(\mathbb R)$ are continuously embedded in the H\"{o}lder spaces $C^{\alpha}$ for all $\alpha < s$.
Here, $[s]$ represents the integer part of $s$ such that $[s] < s$. 

We have a following nice result for $\Lambda^{-s}$ on the spaces $B^{m}_{p,r}(\mathbb R)$.

\begin{prop} \label{Prop_bessel_new} \cite[Prop. 2.78]{MR2768550}
    The operator $\Lambda^{-s}$ is continuous from $B^m_{p,r}(\mathbb R)$ to $B^{m+s}_{p,r}(\mathbb R)$ for $s>0$, $m \in \mathbb R$, and $1 \le p,r \le \infty$. 
\end{prop}
%
\noindent A particular case of the above proposition, which we use later, can be stated as: ``The operator $\Lambda^{-s}$ is bounded from $L^\infty(\mathbb R)$ to ${C}^{\alpha}(\mathbb R)$ for some $\alpha \in (0,1)$ less than $s >0$.''


\section{Properties of solutions} \label{Solution_properties}

This section contains a set of a priori properties associated with the travelling wave solution (say $\varphi$) to \eqref{eq:steady}. These properties will be used in the subsequent section, particularly for constructing solitary waves as the period $P$ tends towards infinity. 
Our approach closely follows that of Ehrnstr\"{o}m et al. \cite{MR4531652}. Most proofs are consistent with those in their work, given the shared common properties between $K$ in \cite{MR4531652} and $K_s$ in our study. For conciseness, several proofs are omitted here and will only be provided when they differ from those in \cite{MR4531652}.

\begin{prop}\label{prop-Phi} \cite{MR4531652}
	Let $\varphi$ be a bounded solution of \eqref{eq:steady} and $\mu \ge 0$ be the wave speed such that $\lim_{x \to \infty} \varphi(x) = \gamma,$
then $\gamma$ will also be the solution of \eqref{eq:steady} with the wave speed $\mu$, and attains the value either $\mu - 1$ or $0$.
\end{prop}

\begin{prop}\label{prop-L2-varphi}
The solution $\varphi \in L^1(-\frac{P}{2}, \frac{P}{2})$ of the steady equation \eqref{eq:steady} belongs to $L^2(-\frac{P}{2}, \frac{P}{2})$ for any $\mu \in \mathbb R$,  $P > 0$ as
\begin{align*}
 (\mu - 1) \int_{-\frac{P}{2}}^{\frac{P}{2}}
 \varphi(x) \, \mathrm dx 
 = \|\varphi \|^{2}_{L^2(-\frac P2, \frac P2)}.
 \end{align*}
Moreover, $\varphi$ has a negative mean if $\mu < 1$ and a positive mean if $\mu > 1$.
\end{prop}
\noindent The above result can be proved by integrating the steady equation \eqref{eq:steady} from $-\frac P2$ to $\frac P2$.

Next proposition predicts the range of the wave speed $\mu$ for the travelling-wave solution to be non-negative and non-decreasing. The proof is similar to the one given in  \cite{MR4531652}, however, we present a simpler version here.

\begin{prop}\label{Mu_ge_1}	
Let $\varphi$ and $\mu$ be same as Prop. \ref{prop-Phi} . Additionally, if $\varphi$ is even, non-negative, non-constant, and non-decreasing on $(-\infty, 0)$, then $\mu \ge 1$.
\end{prop}
\begin{proof}

We prove this proposition by contradiction.		
Assume that $\mu < 1$. Prop. \eqref{prop-Phi} promptly gives us  
   $\gamma = \lim_{x \to -\infty} \varphi(x) = 0$,   
as the assumption $\mu - 1 < 0$ shows that the other possible value  of $\gamma = \mu - 1$ can not hold due to the non-negativity of $\varphi$.
Now, from \eqref{eq:steady}, we have
\begin{align*}
	(\mu - \varphi)\varphi = \Lambda^{-s} \varphi.
\end{align*}
For $x < 0$, the convolution formula gives
\begin{align}
	\big( \mu - \varphi(x)  \big) \varphi(x) 
 =& \int_{\mathbb R} K_s(x - y) \varphi(y) \, \mathrm dy \nonumber\\
=& \int_{\{|y| < |x|\}} K_s(x - y) \varphi(y) \, \mathrm dy 
+ \int_{\{|y| \ge |x|\}} K_s(x - y) \varphi(y) \, \mathrm dy \label{A1 plus A2}
\end{align}
For the first integral of \eqref{A1 plus A2}, since $\varphi$ is non-decreasing on $(-\infty, 0)$, non-negative, and even; and $K_s$ is positive on $\mathbb R \backslash \{0\}$, we get
\begin{align}
\int_{x}^{-x} K_s(x - y) \varphi(y) \, \mathrm dy  & 
	 \ge \varphi(x) \int_0^{-2x} K_s(y) \, \mathrm dy. \label{A1-temp}
\end{align}
For the second integral, choose a variable $\sigma > 0$ such that
\begin{align}
	\frac 12 > \int^\sigma_0 K_s(x)\, \mathrm dx 
 > \mu - \frac 12 + \varepsilon,\quad \varepsilon 
 > 0.\label{s-1}
\end{align}
Again, since $\varphi$ is non-negative and $K_s$ is even and positive, we get
\begin{align}
 \int_{\{|y| \ge |x|\}} K_s(x - y) \varphi(y) \, \mathrm dy 
&\ge 
\int^{x + \sigma}_{x} K_s(x - y) \varphi(y) \, \mathrm dy. \label{A2-temp}
\end{align}
Thus, substituting \eqref{A1-temp} and \eqref{A2-temp} into \eqref{A1 plus A2}, we have
\begin{align*}
 \big( \mu - \varphi(x) \big) \varphi(x) \ge \varphi(x) \int_0^{-2x} K_s(y)\, \mathrm dy + \int^{x+\sigma}_{x} K_s(x-y) \varphi(y)\, \mathrm dy,
\end{align*}
which we can rewrite using monotonicity of $\varphi$ as 
\begin{align*}
\big(   \mu - \varphi(x)   \big) \varphi(x)
& \ge \varphi(x)  \int_0^{-2x} K_s(y)\, \mathrm dy  + \varphi(x) \int^{\sigma}_0 K_s(z) \, \mathrm d z,\\
\implies
 \mu - \varphi(x) 
& \ge  \int_0^{-2x} K_s(y)\, \mathrm dy  +  \int^{\sigma}_0 K_s(z) \, \mathrm d z.
\end{align*}

Now, assume $h(x) := \int_0^{-2x} K_s(y)\, \mathrm dy \leq \frac{1}{2}$ as $\|K_s\|_{L^1(\mathbb R)} = 1$, then there exists a sequence of negative real numbers $\{x_n\}_{n\in \mathbb N}$ with $\lim_{n \to \infty} x_n = - \infty $ such that $ h(x_n) \to \frac{1}{2}$ and $ \phi(x_n) \to 0$ as $n \to \infty$, it follows that
{\begin{align*}
  \mu - \varphi(x_n)  & \ge  h(x_n)  + \big( \mu - \frac 12 + \varepsilon \big)
  \implies \mu   \ge  \frac 12  + \big( \mu - \frac 12 + \varepsilon \big),
\end{align*}
which is a contradiction to itself.} Hence, $\mu \ge 1$ under the assumed conditions.
\end{proof}

The following corollary addresses a challenge arising when examining solutions that may not belong to $L^1(\mathbb{R})$, and offers insight into solutions of \eqref{eq:steady} under certain conditions, establishing a relationship between the solution $\varphi$ and the behavior of convolutions involving $\varphi$ and specific test functions. 

\begin{cor}\label{cor-1} \cite{MR4531652}
Let $\varphi \in L^\infty (\mathbb{R})$ be a solution of \eqref{eq:steady} with a wave speed $\mu \geq 0$, where $\varphi$ is non-negative and possesses a limit as $x$ tends to positive or negative infinity. Then, for every nonzero smooth and compactly supported test function $\psi$ with $\psi \geq 0$, the integral
\begin{align*}
\int_{\mathbb{R}} \psi *
\big( \varphi((\mu -1) - \varphi) \big) \mathrm{d}x = 0,
\end{align*}
holds. Specifically, this implies that only solutions that vanish everywhere have a unit speed.
\end{cor}

\begin{lem}\label{Smoothness}    
Any solution $\varphi$ satisfying $\varphi \leq \frac{\mu}{2}$ to the steady fKdV equation is smooth across every open set where $\varphi$ is less than $\frac{\mu}{2}$.
\end{lem}
\noindent A very brief proof of the this lemma is given in \cite[Lemma 3.5]{Orke-arxiv-22} which is eventually satisfied for all positive $s$ as we need the map $\Lambda^{-s} : L^\infty(\mathbb R) \to {C}^{s}(\mathbb R)$ to be linear and bounded (Prop. \ref{Prop_bessel_new}). Detailed proof can be found in \cite[Theorem 5.1]{EhrnstromWahlen-AIHAN-19}.

\section{The construction of solitary waves}\label{sec-construction}

In this section, we aim to construct a full family of solitary wave with the decay property. For this, we parameterise our solution $\varphi$ which also depends on $\mu$, and consider $(\varphi,\mu)$ as a solution pair. Let $\lambda$ be the \textit{relative wave height} of a solution given by
\begin{align*}
	\varphi_\lambda(0) = \frac{\lambda \mu}{2}, \quad  0 < \lambda \leq 1.
\end{align*}
Corresponding to $\lambda=1$, the wave attains maximum height. The family of crests is placed continuously between zero solution and the highest wave.
%
The following theorem shows the existence of the highest periodic waves for each period $P$ and relative wave height $\lambda$, and
is the combination of results proved in \cite{Orke-arxiv-22}, \cite{Ehrnstrom_Ola_Kristo}, and \cite{Le-AA-22} for the cases $0<s<1$, $s = 1$, and $s>1$, respectively.
{The case $s = 1$ is similar to the case of bidirectional Whitham equation studied by Ehrnstr\"{o}m, Johnson, and Claassen \cite{ehrnstrom2019existence}  where the kernel $K_P$ has logarithmic blowup at $x = 0,$ and is a particular case of model thoroughly studied by Ehrnstr\"{o}m, M\ae hlen, and Varholm \cite{Ehrnstrom_Ola_Kristo}. Section 3 in \cite{Ehrnstrom_Ola_Kristo} comes with few assumptions on the kernel and solution of the considered problem, which aligns well with our case. Whereas Section 4 in \cite{Ehrnstrom_Ola_Kristo} studies logarithmic kernel, its bound and global regularity of highest waves.


Let $\varphi_{P, \lambda}$ is a periodic solution of \eqref{eq:steady} with the periodised kernel $K_{s,P}$.

\begin{thm}\label{thm-1}
For each finite $P > 0$ and each $\lambda \in (0, 1]$,
\begin{enumerate}[label=\roman*.]

\item the steady  equation \eqref{eq:steady} has an even, $P$-periodic solution $\varphi_{P, \lambda} $ with wave speed $\mu_{P, \lambda}$ and relative wave height $\lambda$ such that $\varphi_{P, \lambda} $ is strictly increasing on $(-\frac P2, 0)$.

\item $\varphi_{P, \lambda}$ is smooth on $\mathbb R \backslash P \mathbb Z$. For $\lambda < 1$, $\varphi_{P, \lambda} \in C^\infty (\mathbb R)$.

\item $\varphi_{P, \lambda} \in C^{s} (\mathbb R)$ for $s \in (0,1)$;  $\varphi_{P, \lambda}$ is log-Lipschitz for $s = 1$; and for $s > 1$, $\varphi_{P, \lambda}$ is Lipschitz.  

\item the solutions have sub-critical wave speed $0 < \mu_{P, \lambda} \leq 1$ and are uniformly bounded by
\begin{align}\label{eq4.1}
	\mu_{P, \lambda} - 1 \leq \varphi_{P, \lambda}
\leq \varphi_{P, \lambda}(0) = \frac{\lambda \mu_{P, \lambda}}{2}. 
\end{align}
\end{enumerate}
\end{thm}
The existence of such solutions $\varphi_{P, \lambda}$ when $s \in (0,1)$ is given in Section 3.3 of \cite{Orke-arxiv-22} and is proved by using bifurcation theory. Section 5 of \cite{Le-AA-22} follows the same approach for $s > 1$. The case $s = 1$ for existence of the desired solutions is addressed in \cite[Section 5]{ehrnstrom2019existence}. 
The smoothness of $\varphi_{P, \lambda}$ away from the origin is ensured by Lemma \ref{Smoothness}.
Third part (\textit{iii.}) demonstrates the behavior of solutions near origin. The highest travelling wave has $s-$H\"{o}lder regularity near origin when $s \in (0,1)$ and nicely demonstrated in Theorem 3.7 in \cite{Orke-arxiv-22}. For $s > 1$, Theorem 4.6 in \cite{Le-AA-22} proves Lipschitz continuity of the solution near origin for the equation \eqref{eq:steady}. 
{\cite{Ehrnstrom_Ola_Kristo} deals with the kernel having logarithmic singularity around zero.
We discuss this case here to show that the our case with $s = 1$ falls in the same  settings, however, here we restrict ourselves to the periodic solutions. }  
\begin{proof}
\textit{iii.} (for the case $s=1$)
For the steady periodic solution $\varphi_P$ of \eqref{eq-fKdV}, we have
\begin{align}\label{steady.periodic.eq}
	K_{s}*\varphi_P = \mu \varphi_P - \varphi_P^2.
\end{align}
Let $f(t) = \mu t - t^2$ such that $-\frac 12 f''(t) = 1$. The function $f$ attains a maximum at $\frac \mu 2$, and 
\begin{align*}
    f\big(\frac \mu 2 \big) - f(t) = \big( -\frac 12 f''(t) \big) \big( \frac \mu 2 - t \big)^2 = \big( \frac \mu 2 - t \big)^2.
\end{align*}
Let $\om = \frac \mu 2 - \varphi_P$, then $\om>0$ satisfies the above equation in the following sense
\begin{align}\label{omega}
\om^2(x) & = K_{s}*\om(x) - K_{s}*\om(0) \nonumber\\
         & = \int_{-P/2}^{P/2} K_{s,P}(x-y) \om(y) \,\mathrm dy - \int_{-P/2}^{P/2} K_{s,P}(y) \om(y) \,\mathrm dy \nonumber\\
         & = \int_{0}^{P/2} [ K_{s,P}(x+y) + K_{s,P}(x-y) - 2K_{s,P}(y)] \om(y) \,\mathrm dy  \nonumber\\        
         & = \int_{0}^{P/2} \Delta_x^2 K_{s,P}(y) \om(y) \,\mathrm dy.
\end{align} 
where $\Delta_x^2f(t)$
is the usual second-order difference. 
Since, $K_{s,P}(x)$ comes with logarithmic singularity $\log(1/|x|)$ for $s=1$, we divide \eqref{omega} by $\frac{1}{x^2 \log (1/x)^2}$ and split the integral for $0<x<\frac{P}{2}< \infty$, we have
\begin{align}\label{Big_Intgral}
	\big(g(x)\big)^2 := \left( \frac{\om(x)}{x \log (1/x)} \right)^2 & = \frac{1}{(x \log (1/x))^2}
    \left( \int_{0}^{x} + \int_{x}^{P/2} \right) 
    \Delta_x^2 K_{s,P}(y) \om(y) \,\mathrm dy.
\end{align}
Following the same steps as in \cite[Lemma 4.1]{Ehrnstrom_Ola_Kristo}, we get the bounds for $g(x)$ as follows
\begin{align}\label{Gt_sq}
 o(1)g(x)+ \frac 12 \max_{t\in (0,r] }g(t) \geq g(x)^2 \geq o(1)g(x)+ \frac 12 \min_{t\in (0,r] }g(t),
\end{align}
provided $R_{s,P}'' \in L^1(\frac{-P}{2},\frac{P}{2})$, where $R_{s,P}$ is the regular part of the kernel $K_{s,P}$ given by Remark \eqref{Remark1}. For $x>0$, $ R_{s,P}''(x) = K_{s,P}''(x) - S_{s,P}''(x)$, and
\begin{align*}
     \int_{-P/2}^{P/2} |R_{s,P}''(x)| \,\mathrm dx 
    &= \int_{-P/2}^{P/2} \left| K_{s,P}''(x) - S_{s,P}''(x) \right| \,\mathrm dx\\
    & \leq  \sum_{n\in \mathbb{Z}} \int_{-P/2+nP}^{P/2+nP} \left| K_{s}''(x) - S_{s}''(x) \right| \,\mathrm dx
     = 2\int_{0}^{\infty}  \left| K_{s}''(x) - S_{s}''(x) \right|  \,\mathrm dx, 
\end{align*}
where $S_s(x)$ is the singular part in $K_s$ and is {$\frac{1}{\pi} \log (\frac{1}{|x|})$} near the origin when $s=1$. $K_{s}(x) = \frac{1}{\pi} G_0(x)$, where $G_0(x) = \int_1^{\infty} \frac{e^{-xt}}{\sqrt{t^2-1}}\,\mathrm dt$ is the modified Bessel function of second kind for $s=1$, and $G_0''(x) = \int_1^{\infty} e^{-xt} \frac{t^2}{\sqrt{t^2-1}}\,\mathrm dt,$ this gives
\begin{align*}
     \int_{-P/2}^{P/2} | R_{s,P}''(x) | \,\mathrm dx 
   & \leq \frac{2}{\pi} \int_{0}^{\infty} \big| G_{0}''(x) - \frac{1}{x^2} \big| \,\mathrm dx\\
    & = \frac{2}{\pi} \int_{0}^{\infty} \left| \int_1^{\infty} e^{-xt} \frac{t^2}{\sqrt{t^2-1}}\,\mathrm dt - \int_0^{\infty} te^{-xt} \,\mathrm dt \right| \,\mathrm dx\\
    & \leq \frac{2}{\pi} \int_{0}^{\infty} \left(  \int_1^{\infty} e^{-xt} \left( \frac{t^2}{\sqrt{t^2-1}}-t\right)\,\mathrm dt + \int_0^{1} te^{-xt} \,\mathrm dt \right) \,\mathrm dx\\
    & = \frac{2}{\pi} \left( \int_1^{\infty} \left( \frac{t}{\sqrt{t^2-1}}-1  \right)\,\mathrm dt + \int_0^{1} 1 \,\mathrm dt \right) = \frac{4}{\pi}.
\end{align*}

Since $\varphi_P$ is a periodic, even solution of the steady equation \eqref{steady.periodic.eq} such that it is  decreasing and smooth on an interval $(0,r), \ r>0$, and $\varphi_P(0) = \frac \mu 2$, then from \eqref{Gt_sq}
\begin{align}\label{AsymBehavior1}
	\varphi_P = \frac{\mu}{2} - \left(  \frac{1}{2} + o(1) \right) x \log(1/x) \quad \text{as} \quad x \searrow 0.
\end{align}
As $\varphi_P$ is continuously differentiable and decreasing in $(0,P/2)$, following the same idea as above, we have
\begin{align}\label{AsymBehavior2}
	\varphi_P' =  - \left(  \frac{1}{2} + o(1) \right) \log(1/x) \quad \text{as} \quad x \searrow 0.
\end{align}
A detailed proof of \eqref{AsymBehavior2} can be found in \cite[Prop. 4.4]{Ehrnstrom_Ola_Kristo}.

\noindent Lemma 4.6 of \cite{Ehrnstrom_Ola_Kristo} proves that for an absolutely continuous function $\varphi$ on an open interval $I$ containing zero, and for a modulus of continuity $w$ such that 
\begin{align}\label{eq4.10}
    \text{ess} \lim_{x\to0} \text{sup}  \frac{|\varphi'(x)|}{w'(|x|)} < \infty,
\end{align}
 $\varphi$ admits $Mw$ as a modulus of continuity on the same interval $I$, i.e.,
 \begin{align*}
    |\varphi(x)-\varphi(y)| \leq M w(|x-y|), \quad \text{for all} \quad x,y \in I,
\end{align*}
for a positive constant $M$, depending only on \eqref{eq4.10} and {properties of $w$}.
Since $\varphi_P$ is {an absolutely continuous, even, and $P$-periodic solution of \eqref{steady.periodic.eq} given by \eqref{AsymBehavior1} and \eqref{AsymBehavior2} which is smooth and decreasing on a small interval $(0,r)$ then by choosing $w(x) = x\log(1+1/x)$, we have
\begin{align}\label{log-Lips}
    |\varphi_P(x)-\varphi_P(y)| \leq M |x-y| \log \left( 1 + \frac{1}{|x-y|} \right),
\end{align}
on an open interval $I$ containing zero.}
Since the solution is uniformly bounded and smooth away from the origin, with estimates only depending on the \(L_\infty\)-norm of the solution (Lemma \eqref{Smoothness}), we get \eqref{log-Lips} as a bound, uniformly in $P$.

\textit{iv.} For the inequality \eqref{thm-1}, we assume that there exists such solutions described by (\textit{i.}), let $\varphi_1$ be the super-solution
of \eqref{eq:steady}, i.e.  $-\mu \varphi_1  + \Lambda^{-s}\varphi_1 + \varphi_1^2 \leq 0 $. We have,
\begin{align*}
-\mu \varphi_1  + \varphi_1^2 & \leq - \Lambda^{-s}\varphi_1  \\
\implies \big(- \frac{\mu}{2} + \varphi_1 \big)^2  \leq  \frac{\mu^2}{4} & - \Lambda^{-s}\varphi_1   \leq  \frac{\mu^2}{4} - \inf\varphi_1,
\end{align*}	
as $\Lambda^{-s}$ is a monotone operator (Prop. \ref{Prop_bessel_2}) and $\Lambda^{-s}c=c$ for positive constant $c$. In particular, 
\begin{align*}
\big(- \frac{\mu}{2} + \inf \varphi_1 \big)^2  \leq  \frac{\mu^2}{4} - \inf\varphi_1 \implies \inf\varphi_1 \big(\inf \varphi_1 - (\mu -1) \big) \leq 0.
\end{align*}
If $\inf \varphi_1 \ge 0$, this gives $\inf \varphi_1 \leq \mu -1$, and this implies, $\mu \ge 1.$ Hence, for the case $\mu \leq 1$, $\mu -1 \leq \inf \varphi_1$. Any P-periodic solution $\varphi_P$ of equation \eqref{eq:steady} must satisfy the inequality obtained for the super-solution i.e. $\mu -1 \leq \inf \varphi_P$.

\noindent From equation \eqref{eq:steady}, $ \Lambda^{-s}\varphi_P' = (\mu - 2 \varphi_P) \varphi_P' $, and since, $ \varphi_P' $ satisfies all properties of function described by Prop. \ref{Prop_bessel_2} (\textit{ii}), hence, $ \Lambda^{-s}\varphi_P' > 0$ implies that $ \mu - 2 \varphi_P > 0 $ and we have, $  \varphi_P <  \mu/2 $.
Let us parameterize the family by assuming  $\lambda$ as a relative wave height of solution such that  $\varphi_{P, \lambda}(0) = \frac{\lambda \mu_{P, \lambda}}{2}, \lambda \in (0,1]$, and combining the results, we have \eqref{eq4.1}.
\end{proof}

In the following theorem, we shall establish limit for the solution as the period $P$ approaches infinity. 
%
\begin{thm}\label{prop-delta}
	Let $(\varphi_{P, \lambda}, \mu_{P,\lambda})$
be as in the previous theorem then for all $P \ge 1$, $\lambda \in (0, 1]$, there exists a positive constant $\delta$ such that for $|x| \leq \delta$, $\varphi_{P, \lambda}$ will satisfy the following inequalities depending upon the value of $s$,
\begin{align}
	\frac{\mu_{P, \lambda}}{2} - \varphi_{P, \lambda}(x) \ge \delta |x|^{\sigma}, 
\end{align} 
where $\sigma = s$ when $s \in (0,1)$, and $\sigma = 1 $ for $  s > 1$.
For $s = 1$, we have
\begin{align}\label{logbound}
	\frac{\mu_{P, \lambda}}{2} - \varphi_{P, \lambda}(x) \ge \delta |x \log |x||.
\end{align} 
%
%
\end{thm}
\begin{proof}

We shall only consider the case $x < 0$ as $\varphi_{P, \lambda}$ is even. Now, the periodicity and evenness of $K_{s, P}$ and $\varphi_{P, \lambda}$ give
\begin{align}
	&(\Lambda^{-s}\varphi_{P, \lambda})(x + h)-
(\Lambda^{-s}\varphi_{P, \lambda})(x - h)\nonumber\\
=&	\int_{-\frac{P}{2}}^{0}
\Big(
\big(K_{s, P} (y - x) - 
K_{s, P} (y + x)\big)
\big(
\varphi_{P, \lambda}(y + h)
-
\varphi_{P, \lambda}(y - h)
\big)
\Big)
\,
\mathrm dy. \label{temp-integral}
\end{align}
By Lemma \ref{lem-K-1} and Lemma \ref{lem-K-2}, we have
\begin{align*}
    K_{s, P} (y - x) - K_{s, P} (y + x) > 0\quad \text{for}\quad -\frac{P}{2} < x, y < 0.
\end{align*}
By Theorem \ref{thm-1} (\textit{i.}), 
\begin{align*} 
\varphi_{P, \lambda}(y + h)
-
\varphi_{P, \lambda}(y - h) \ge 0 \quad \text{for} \quad -\frac{P}{2} < y < 0,\, 0 < h < \frac{P}{2}.
\end{align*}
For $-\frac{P}{2} < x < 0$, by \eqref{temp-integral} and Fatou's lemma, we get
\begin{align}
	\Big(
\frac{\mu_{P, \lambda}}{2} - \varphi_{P, \lambda}(x)\Big)   \varphi'_{P, \lambda}(x)
&=
\lim_{h \to 0}
\frac
{(\Lambda^{-s} \varphi_{P, \lambda})(x + h) 
- 
(\Lambda^{-s} \varphi_{P, \lambda})(x - h)}
{4h} \nonumber\\
&\ge 
\frac{1}{2}
\int_{-\frac{P}{2}}^{0}
\big(K_{s, P} (y - x) - 
K_{s, P} (y + x)\big)
\varphi_{P, \lambda}'(y)
\,
\mathrm dy.\label{temp-1}
\end{align}

\textbf{Case 1} when $s \in (0,1)$: 
Let's fix two points say $p_1$, $p_2$ from the interval $(-\frac{P}{2}, 0)$ such that $ p_2 < p_1 $. Choose $x \in (p_2, p_1)$ and $\varsigma \in (-\frac{P}{2}, p_2]$. Then from  \eqref{temp-1}, and 
the monotonicity of $\varphi_{P, \lambda}$ on $(-\frac{P}{2}, 0)$,
we have	
\begin{align*}
\Big(
\frac{\mu_{P, \lambda}}{2}
-
\varphi_{P, \lambda}(\varsigma)
\Big)
\varphi'_{P, \lambda} (x)	
&\ge 
\Big(
\frac{\mu_{P, \lambda}}{2}
-
\varphi_{P, \lambda}(x)
\Big)
\varphi'_{P, \lambda} (x)\\
&\ge
\frac 12 \int_{-\frac{P}{2}}^{0}
\Big(
K_{s, P}(y - x) - K_{s, P}(y + x)
\Big)\varphi'_{P, \lambda}(y)\,
\mathrm dy\\
&\ge 
\frac 12 \int_{p_2}^{p_1}
\Big(
K_{s, P}(y - x) - K_{s, P}(y + x)
\Big)\varphi'_{P, \lambda}(y)\,
\mathrm dy\\
&=
\frac 12 \int_{p_2}^{p_1}
(-2x)
K_{s, P}'(y + \tau)
\varphi'_{P, \lambda}(y)
\,\mathrm dy\\
&\ge -p_1 K_{s, P}'(2p_2)
\Big(
\varphi_{P, \lambda}(p_1) - \varphi_{P, \lambda}(p_2)
\Big),
\end{align*}
where $|\tau| < |x|$ follows from 
the convexity of $K_{s, P}$ on $(-\frac{P}{2}, 0)$
and the mean value theorem. Integrating from $p_2$ to $p_1$ on $x$ and dividing by $\big(\varphi_{P, \lambda}(p_1) - \varphi_{P, \lambda}(p_2)\big)$ gives
\begin{align}\label{eq4.6}
	\frac{\mu_{P, \lambda}}{2} - \varphi_{P, \lambda}(\varsigma)
	\ge -p_1 K_{s, P}'(2p_2)(p_1 - p_2).
\end{align}
Now, 
from  \eqref{eq:periodisedKernel} and Lemma \ref{lem-K-1}, we have
 \begin{align*}
     K'_{s, P}(x) & = K'_s(x) + \sum_{n \in \mathbb Z \backslash \{ 0 \}} K'_s(x + nP)\\
& \leq c_1 (s-1) |x|^{s-2} + O(|x|^s) + c_2 \sum_{n \in \mathbb Z \backslash \{ 0 \}} |x+nP|  e^{-|x+nP|},
 \end{align*}
here
$c_1$ and $c_2$ are positive constants depending on $s$ only, and $x$ is certainly negative. 
By rewriting the above relation as
  \begin{align*}
    - K'_{s, P}(x) & \geq - c_1 (s-1) |x|^{s-2} - O(|x|^s) - c_2 \sum_{n \in \mathbb Z \backslash \{ 0 \}} |x+nP|e^{-|x+nP|}.
 \end{align*}
 For 
 $-\frac{P}{4}<p_2<0, $
   \begin{align*}
 - K'_{s, P}(2p_2) & \geq - c_3 |p_2|^{s-2} - O(|2p_2|^s) - c_2 \sum_{n \in \mathbb Z \backslash \{ 0 \}} |2p_2+nP|e^{-|2p_2+nP|} \\
   & \geq - c_3|p_2|^{s-2} - O(|2p_2|^s) - c_2 \sum_{n \in \mathbb Z \backslash \{ 0 \}} \left(\frac{2|n|+1}{2}\right)Pe^{-\left(\frac{2|n|+1}{2}\right)P} \\
   & \geq - c_3|p_2|^{s-2},
 \end{align*}
 here $c_3$ is another positive constant depending on $s$. 
 %
By using the above inequality in (\ref{eq4.6}), 
 \begin{align*}
	\frac{\mu_{P, \lambda}}{2} - \varphi_{P, \lambda}(\varsigma)
	\ge -p_1 c_3 |p_2|^{s-2} (p_1 - p_2).
\end{align*}
and taking 
$\varsigma = p_2 = x$ and $p_1 = \frac{x}{2}$ with $-\frac{P}{4} < x < 0$, we have
\begin{align}
\frac{\mu_{P, \lambda}}{2} - \varphi_{P, \lambda}(x)
\ge c'|x|^{s}. \label{mu-minus-phi}	
\end{align}	
 Now, by picking $0 < \delta < c'$ small enough, we get
\begin{align*}
 \frac{\mu_{P, \lambda}}{2} - \varphi_{P, \lambda}(x)
\ge
\delta |x|^{s},    
\end{align*}
uniformly for $P \ge 1$, $0 < \lambda \leq 1$ and 
$|x| < \delta$ by continuity.

 \textbf{Case 2} when $s >1$: The idea of proof given here is taken from \cite{Bruell-Dhara-2018}. 
 From Theorem \ref{thm-1}, we know that the $ \varphi_{P, \lambda} $ is Lipschitz at crest, that is, 
 \begin{align}\label{upper_bound}
\mu - \varphi_{P, \lambda} \lesssim |x| \quad \text{for} \quad |x| \ll 1,  
\end{align}
which sets the upper bound for  $(\mu - \varphi_{P, \lambda} )$.
To establish the lower bound for $|x| \ll 1$,  we divide \eqref{temp-1} by \eqref{upper_bound}, 
\begin{align}\label{case2.2}
  \varphi'_{P, \lambda}(x)
\gtrsim 
\frac{1}{2}
\int_{-\frac{P}{2}}^{0}  \frac{K_{s, P} (y - x) -  K_{s, P} (y + x)}{|x|} \varphi_{P, \lambda}'(y)
\,
\mathrm dy.
\end{align}
Since, for some $y \in (-\frac{p}{2},0)$, we can write
\begin{align*}
\lim_{x \to 0^-} \frac{K_{s, P} (y - x) -  K_{s, P} (y + x)}{|x|} = 2 K'_{s, P} (y).
\end{align*}
By substituting this value in \eqref{case2.2} and applying Fatou's lemma, we get  
\begin{align}\label{lower_bound_const}
\liminf_{x \to 0}  \varphi'_{P, \lambda}(x)
\gtrsim 
\int_{-\frac{P}{2}}^{0}  K'_{s, P} (y) \varphi_{P, \lambda}'(y)
\,
\mathrm dy,
\end{align}
{ As both $ K'_{s, P} $ and $ \varphi_{P, \lambda}'$ are continuous and bounded in $(-\frac{p}{2},0)$, the integral in the left hand side is a constant.}
%
Now, by applying the Mean Value Theorem on  $\varphi_{P, \lambda} $ for interval $(x,0)$ for $x<0$,
\begin{align*}
\frac{\varphi_{P, \lambda}(0)-\varphi_{P, \lambda}(x)}{|x|} = \varphi'_{P, \lambda}(\xi) \quad \text{for some} \quad |\xi| \ll 1.
\end{align*}
Since, $\varphi_{P, \lambda}(0) = \frac{\lambda \mu_{P, \lambda}}{2}$, or $\varphi_{P, \lambda}(0) \leq \frac{ \mu_{P, \lambda}}{2}$, by placing the lower bound from \eqref{lower_bound_const}, we have the desired result.

 \textbf{Case 3} when $s = 1$: We discussed this case in the previous theorem, which provides the stronger version, however, to get the uniform upper bound let's restrict ourselves to the one inequality from \eqref{Gt_sq} only, 
 \begin{align}
g(x)^2 \geq o(1)g(x)+ \frac 12 \min_{t\in (0,r] }g(t),
\end{align}
here $g(x)=\frac{ \frac{\mu_{P,\lambda}}{2} - \varphi_{P, \lambda}(x)}{x\log(1/x)}$ is positive and monotonic in $(0,r]$, for any small $\delta > 0,$ 
 \begin{align*}
     \frac{ \frac{\mu_{P,\lambda}}{2} - \varphi_{P, \lambda}(x)}{|x\log(1/x)|} \geq \delta
 \end{align*}
$\frac{\mu_{P,\lambda}}{2}-\varphi_{P, \lambda}(x) \in C([0,P])$ is bounded and non-negative, providing $\varphi_{P, \lambda}$ is periodic solution of \eqref{eq:steady}, and $\varphi_{P, \lambda}(0)= \frac{\lambda \mu_{P, \lambda}}{2}$. Also, $|x\log(1/x)|$ in increasing function in $(0,r]$, \eqref{logbound} holds true for $s=1$.
\end{proof}

We now construct a sequence $\{ (\varphi_{P_n, \lambda}, \mu_{P_n, \lambda})  \}$ of periodic solution pairs obtained from the previous two theorems. We shall show that there exists a sub-sequence converging to a non-periodic solution which might not decay to $0$ as $P_n \to \infty$. The limiting solution will eventually be the solution we desire and after Galilean transform, it gives the solitary wave solution.

\begin{thm}\label{main-thm}
The sequence  $\{ (\varphi_{P_n, \lambda}, \mu_{P_n, \lambda})  \}$ of solution pairs from Theorem \ref{thm-1} has a subsequence converging locally uniformly to a solution pair 
$(\phi_\lambda, \vartheta_\lambda) \in C(\mathbb R) \times (0,1]$ as $P_n \to \infty$
with relative height $\lambda \in(0, 1] $ such that 
\begin{align*}
	\vartheta_\lambda - 1 \leq \phi_\lambda \leq \frac{\lambda \vartheta_\lambda}{2}.
\end{align*}
This non-periodic limiting solution $\phi_\lambda$ is bounded, non-constant, even, and strictly increasing on $(-\infty, 0)$. Moreover, $\phi_\lambda$ is  smooth
on $\mathbb R$ for $\lambda \in (0, 1)$ and $\vartheta_\lambda > 0$.

Further, the Galilean transformation 
 \begin{align}\label{Gal_Trans}
     \Phi_\lambda = \phi_\lambda + 1 - \vartheta_\lambda,~~~~~~~ \mu_\lambda = 2 - \vartheta_\lambda,
\end{align}
 provide a solution $  \Phi_\lambda $ of \eqref{eq:steady} which satisfies all the properties of $\phi_\lambda$ and gives the desired solitary wave solution providing $\lim_{x \to \pm \infty} \Phi_\lambda(x) = 0$ with supercritical wave speed $ \mu_\lambda \in (1,2) $ such that
\begin{align*}
	0 < \Phi_\lambda \leq \mu_\lambda - 1 + \lambda \big( 1 - \frac{\mu_\lambda}{2} \big). 
\end{align*}
$\Phi_\lambda$ is smooth everywhere for $\lambda \in (0,1)$ and for $\lambda = 1$, {{$\Phi_\lambda$}} is smooth on $\mathbb R \backslash \{0\} $ such that $\Phi_\lambda \in C^s$ for $s \in (0,1)$, $\Phi_\lambda $ is log-Lipschitz for $s=1$ and Lipschitz for $s>1$. 
\end{thm}

\begin{proof}
The approach of the proof is similar to the one presented in \cite{MR4531652}, however, here we deal with the operator $\Lambda^{-s}$ and different cases of $s > 0$. We assume that there exists a sub-sequence $\{P_n\}_{n}$ such that
	$\lim_{P_n \to \infty} \mu_{P_n, \lambda} = \vartheta_\lambda.$
It can be easily shown that
\begin{align*}
[\varphi_{P, \lambda}(x) - \varphi_{P, \lambda}(y)]^2 
\leq
\big|
\Lambda^{-s}\varphi_{P, \lambda}(x) - \Lambda^{-s}\varphi_{P, \lambda}(y)
\big|,
\end{align*}
for all $x, y\in \mathbb R$ . From the Prop. \ref{Prop_bessel_new} and the following remark,  $\Lambda^{-s}$ is bounded from $L^\infty(\mathbb R)$ to ${C}^{\alpha}(\mathbb R)$ for any $\alpha \in (0,1)$,
\begin{align*}
\frac{\big|
\Lambda^{-s}\varphi_{P, \lambda}(x) - \Lambda^{-s}\varphi_{P, \lambda}(y)
\big|}{|x-y|^{\alpha}} 
\leq \| \Lambda^{-s}\varphi_{P, \lambda} \|_{{C}^{\alpha}(\mathbb R)}
%
%
\leq \| \Lambda^{-s} \|_{\mathcal{L}(L^{\infty}(\mathbb R),{C}^{\alpha}(\mathbb R))} \| \varphi_{P, \lambda} \|_{{\infty}}
\end{align*}
for all $x, y \in \R$ which shows that $\{\varphi_{P_n, \lambda}\}_n$ is uniformly bounded in 
$C^{\alpha/2}(\mathbb R)$, for some $\alpha \in (0,1)$.
Since $\{\varphi_{P_n, \lambda}\}_n$ is equicontinuous, by Arzel\`a--Ascoli theorem, there exists a subsequence converging locally uniformly to a function $\phi_\lambda \in C(\mathbb R)$.
Also, $|\varphi_{P_n, \lambda}| \leq 1$ and $\|K_s\|_{L^1(\mathbb R)} = 1$, 
Lebesgue's dominated convergence theorem gives
\begin{align*}
	\Lambda^{-s} \varphi_{P_n, \lambda} (x) = K_s * \varphi_{P_n, \lambda}(x)
\to K_s * \phi_\lambda(x) = \Lambda^{-s}  \phi_\lambda(x)
\end{align*}
as $P_n \to \infty$ for all $x \in \mathbb R$. Thus, $\phi_\lambda$ solves \eqref{eq:steady} with wave speed $\vartheta_\lambda$.

 From Theorem \ref{prop-delta}, we have uniform bound for the desired solution $\phi_\lambda(x)$ near the origin as
\begin{eqnarray}\label{final_bound}
   \phi_\lambda(x) \leq \left\{\begin{array}{lr} \frac{\vartheta_\lambda}{2} - \delta |x|^{s}, & 0<s<1\\
    \frac{\vartheta_\lambda}{2} - \delta |x \log|x|| , & s = 1 \\
     \frac{\vartheta_\lambda}{2} - \delta |x|, & s>1
    \end{array}\right..
\end{eqnarray}

\noindent Recall that the only constant solutions to \eqref{eq:steady} is $\phi_\lambda = 0$ and $\phi_\lambda = \vartheta_\lambda - 1$. 
When $\phi_\lambda = 0$, $\phi_\lambda(0) = \frac{\lambda\vartheta_\lambda}{2}$ gives
$\vartheta_\lambda = 0$,
and this will contradict to \eqref{final_bound}.
When 
    $\phi_\lambda = \vartheta_\lambda - 1$, 
we have, 
    $ \lambda \frac{\vartheta_\lambda}{2} = \vartheta_\lambda - 1$, 
i.e, 
	$\vartheta_\lambda = \frac{2}{2 - \lambda}$.
Since,
$0 < \lambda \leq 1$, this contradicts with $0 < \vartheta_\lambda \leq 1$. Hence, $\phi_\lambda$ is non-constant. 

The evenness of $\phi_\lambda$ is inherited from $\{\varphi_{P_n, \lambda}\}_n$, and it is easy to show that $\phi_\lambda$ is strictly increasing on $(-\infty, 0)$, one can follow the same steps as in \cite[Theorem 4.3]{MR4531652}.
The smoothness of $\phi_\lambda$ on $\mathbb R$ for $0 < \lambda < 1$ is followed from Lemma \ref{Smoothness}.  The property of $\phi_\lambda$ of being non-constant with monotonicity shows that it is  non-periodic.
Again, since $\phi_\lambda$ is bounded, Proposition (\ref{prop-Phi}) tells $\lim_{x \to \infty} \phi_\lambda(x)$ could be $0$ or $\vartheta_\lambda - 1$. Assuming $\lim_{x \to \infty} \phi_\lambda(x) = 0$ immediately gives $\phi_\lambda$ is non-negative. Now, from Proposition (\ref{Mu_ge_1}) and our assumption here, it yields $\vartheta_\lambda = 1$. However, from Corollary (\ref{cor-1}), only trivial solutions can have unit wave speed. Hence, $\lim_{x \to \infty} \phi_\lambda(x) = \vartheta_\lambda - 1 < 0.$ 

The limiting non-periodic solution inherits its regularity from the corresponding periodic wave. As a particular case when $s=1$, we may use the triangle inequality and the earlier estimate (see \eqref{log-Lips}) to write
\begin{align*}
     \big| \phi_\lambda(x) - \phi_\lambda(y) \big| 
    & \leq \big| \phi_\lambda(x) - \phi_{\lambda,P}(x) \big|
    + \big| \phi_\lambda(y) - \phi_{\lambda,P}(y) \big| + \big| \phi_{\lambda,P}(x) - \phi_{\lambda,P}(y) \big|\\
    &\leq \big| \phi_\lambda(x) - \phi_{\lambda,P}(x) \big|
    + \big| \phi_\lambda(y) - \phi_{\lambda,P}(y) \big| +  M |x-y| \log \left( 1 + \frac{1}{|x-y|} \right).
\end{align*}
Now fix a compact interval \(K\) and consider \(x,y \in K\). By letting \(P \to \infty\), because of the uniform convergence, 
the two first terms in the right-hand side converges to \(0\). Therefore,
\begin{align*}
     \big| \phi_\lambda(x) - \phi_\lambda(y) \big| 
    \leq M |x-y| \log \left( 1 + \frac{1}{|x-y|} \right),
\end{align*}
for all \(x,y \in K\). But since \(K\) is arbitrary, this estimate holds for all \(x,y \in \mathbb{R}\).

Now, we only have to exclude the negative tails to obtain the positive solutions say $\Phi_\lambda $ as $x$ approaches $\infty$. For this, we apply the Galilean transformation \eqref{Gal_Trans} which gives the solitary wave solution with wave speed $\mu_\lambda \in (1,2)$. We eliminate the possibility that $\mu_\lambda = 2$ as then again from Corollary (\ref{cor-1}), it would lead us to constant solution $\Phi_\lambda \equiv 0$ or $\Phi_\lambda \equiv 1$. All other properties of $\Phi_\lambda $ are eventually satisfied including smoothness. 
\end{proof}

In the following subsection, we discuss the decay rate of solitary waves for the equation \eqref{eq-fKdV} in accordance with the work of Arnesen \cite{MR4324298}. Several other recent works in this area for different equations can be found in \cite{bruell2017symmetry,eychenne2023decay,pei2020exponential}.

\subsection{Exponential decay}

It has been shown that the decay rate of solitary-wave solutions is related to the decay rate of the associated convolution kernel. Rewriting the steady equation as
\eqref{eq:steady} as
\begin{align*}
    \varphi(\mu-\varphi)=H_\mu * \varphi^2, \quad \text{where} \quad H_\mu = \mathcal{F}^{-1}\left(\frac{m}{\mu-m}\right),
\end{align*}
shifts the focus to $H_\mu$. The exponential decay of $H_\mu$ implies the exponential decay of the solution, with the exact rate depending on the wave speed $\mu$ and $s$.
For this section, we use $m(\xi) = (1+\xi^2)^{-s/2}$ for conciseness.

Arnesen \cite{MR4324298} studies the exact decay rate and symmetry of solitary waves for a class of non-local dispersive wave equations $u_t + (\mathcal{L}u + G(u))_x = 0$, reduced to the steady form 
 \begin{align}\label{Mathias.Eq}
     -\mu \varphi + \mathcal{L}\varphi + G(\varphi) = 0.
 \end{align}
Here $\mathcal{L}$ is a Fourier multiplier operator with a symbol $m$ defined by
   $ \widehat{\mathcal{L} f}(\xi) = m(\xi) \widehat f (\xi),$
and $G$ is a nonlinear function. 
Under certain conditions on $m$ and $G$ (as detailed in the theorem below), solitary waves for this class of equations exhibit exponential decay, have only one crest, and are, in general, symmetric. We provide only the main theorem here, and for the deeper insights one can refer \cite{MR4324298}.
\begin{thm}\cite{MR4324298}\label{Mathias_theorem}
    Let $m(\xi)$ be an even, real analytic function, and bounded above by $\mu$ for all $\xi \in \R$. Let $G:\R \to \R$ be bounded on every compact set, and satisfy $|G(u)|\ls |u|^t$ for all small values of $u$ and for some $t>1$. 
    If the function $m(\xi)$ follows the inequality
    \begin{align*}
        |m^{(k)}(\xi)| \leq C_k(1+|\xi|)^{-s-k}, \quad k = 0, 1,2, \dots,
    \end{align*}
    for some $s>0$, where $C_k$ are positive constants depending on $k$ such that
    \begin{align}\label{coeff.1}
    \lim_{k \to \infty} \frac{C_{k+1}/(k+1)!}{C_{k}/k!} = l, \quad \text{for any} \quad l \geq 0,
\end{align}
%
then for a non-trivial solution $u \in L^\infty(\mathbb R)$ with  $\lim_{|x| \to \infty} u(x) = 0$, 
there exists a number $\delta_{\mu}>0$ depending upon $m$ and $\mu$ such that
\begin{align*}
    e^{\delta |.|} u(.) \in L^1 \cap L^{\infty}(\R),
\end{align*}
for all $ \delta \in (0, \delta_{\mu}).$
\end{thm}

Our case with the fKdV equation with \emph{inhomogeneous} symbol $(1+\xi^2)^{-\frac{s}{2}}$ aligns with the assumptions made in Theorem \eqref{Mathias_theorem}.
We have $G(u) = u^2$, which immediately fulfils all the requirements on $G$,  
and the symbol  $m(\xi) = (1 + \xi^2)^{-\frac{s}{2}}$ is even, real analytic, and bounded above by $\mu$ (as $\mu$ is the supercritical wave speed obtained in Theorem \eqref{main-thm}). We only need to prove the limit \eqref{coeff.1}. 
For this, we use Faà di Bruno's formula to calculate the following derivatives 
\begin{align*}
    D^k_\xi (1 + \xi^2)^{-\frac{s}{2}} &= \sum_{j=0}^{[\frac{k}{2}]} \frac{(-1)^{k-j} 2^{k-2j} k!}{j! (k-2j)!}
    \left[  \frac{s}{2} \left(\frac{s}{2}+1\right) \dots \left(\frac{s}{2}+k-j-1\right)  \right] \xi^{k-2j} (1 + \xi^2)^{-\frac{s}{2}-(k-j)}\\
    &\leq (-1)^k 2^{2k} k! \text{M}_{s,k} (1 + \xi^2)^{-\left(\frac{s+k}{2}\right)},
\end{align*}
where $\text{M}_{s,k}$ is the uniform upper bound (depending on $s$ and $k$) for the series
\begin{align*}
    \sum_{j=0}^{[\frac{k}{2}]} \frac{(-1)^{j} }{j! (k-2j)!} \frac{1}{2^{k+2j}}
       \left[ \frac{s}{2} \left(\frac{s}{2}+1\right) \dots \left(\frac{s}{2}+k-j-1\right) \right].
\end{align*}
This leads 
\begin{align*}
    |D^k_\xi (1 + \xi^2)^{-\frac{s}{2}}| \leq   2^{2k} k! C_0 \text{M}_{s,k} (1 + |\xi|)^{-s-k},
\end{align*}
therefore, the $C_k$'s are $2^{2k} k! C_0 \text{M}_{s,k}$, and we have \eqref{coeff.1}. 

{Now, $m(z)$ is analytic for all $z \in \C\backslash\{-i,i\}$ and along the imaginary axis, we have $m(ib)=(1-b^2)^{-s/2}$ which is even and a bijection from $[0,1)$ to $ [1,\infty)$. This implies that there exists $\delta_{\mu,s} \in (0,1)$  satisfying $(1-\delta_{\mu,s}^2)^{-s/2}=\mu$ for $\mu>1$ such that
\begin{align}\label{H-mu}
    H_\mu(x) = e^{\delta_{\mu,s} |x|}\left( v(x) + \frac{2\sqrt{\pi}\mu (1-\delta_{\mu,s}^2)^{\frac{s+2}{2}}}{s \delta_{\mu,s}} \right)
\end{align}
for some even function $v \in L^p(\{x\in\R:|x|\leq1\})$ for all $P\in(0,\infty)$ with
\begin{eqnarray*}
    v(x) \simeq_s \left\{\begin{array}{lr} |x|^{s-1}, & 0<s<1\\
    \log \frac{1 }{|x|}, & s = 1 \\
     1, & s>1
    \end{array}\right., \quad \text{for} \quad |x| \leq 1.
\end{eqnarray*}
We need the expression \eqref{H-mu} for $H_\mu$ to prove the Theorem \eqref{Mathias_theorem} and to find $\delta_{\mu,s}$ (see \cite[Lemma 3.4]{MR4324298}).
And here we conclude that the obtained solitary solutions $\Phi_\lambda$ of \eqref{eq:steady} with wave speeds $\mu_\lambda \in (1,2)$ and $\lambda \in (0,1)$ such that $\lim_{x \to \pm \infty} \Phi_\lambda(x) = 0$ and $0 \leq \Phi_\lambda \leq \mu_\lambda - 1 + \lambda \big( 1 - \frac{\mu_\lambda}{2} \big)$, 
decay exponentially with
\begin{align*}
    e^{\delta |.|} \Phi_\lambda(.) \in L^1 \cap L^{\infty}(\R),
\end{align*}
for all $ \delta \in (0, \delta_{\mu,s})$. 
}

\section*{Acknowledgments}
SY is supported by the ERCIM ‘Alain Bensoussan’ Fellowship Programme at NTNU, Trondheim. The work is supported in part by the project IMod--Partial differential equations, statistics and data: An interdisciplinary approach to data-based modelling, project number 325114, from the Research Council of Norway. The authors express their gratitude to Mats Ehrnström and Douglas Seth for their continuous support and insightful discussions. The authors also thank Gabriele Brüll for discussions regarding the decay rate, and Kristoffer Varholm for his willingness to explain the mechanics behind the work \cite{Ehrnstrom_Ola_Kristo}.

\bibliographystyle{abbrv}
\bibliography{wave-bibliography}

\end{document}